\documentclass[11pt,letterpaper]{amsart}
\usepackage{graphicx}
\usepackage{amsmath, amssymb, amsthm, tikz, mathrsfs,verbatim,tensor,pstricks}
\usepackage[left=1in,top=1in,right=1in]{geometry}
\usepackage{setspace}
\usepackage{enumerate}

\usepackage{amsfonts}
\usepackage{latexsym}
\usepackage{graphics}

\usepackage{bbm}

\usepackage[sort,nocompress]{cite}

\usepackage{etoolbox}
\patchcmd{\section}{\scshape}{\bfseries}{}{}
\makeatletter
\renewcommand{\@secnumfont}{\bfseries}
\makeatother

\newtheorem{theorem}{Theorem}
\newtheorem{question}{Question}

\theoremstyle{definition}

\theoremstyle{plain}

\newcommand{\beqlbl}{\begin{equation}}
\newcommand{\eeqlbl}{\end{equation}}

\newcommand{\R}{\ensuremath{\mathbb{R}} }

\newcommand{\Z}{\ensuremath{\mathbb{Z}} }

\renewcommand{\P}{\mathbb{P}}

\newcommand{\E}{\mathbb{E}}

\newcommand{\bP}{\mathbf{P}}

\newcommand{\Var}{\mathrm{Var}}

\usepackage{hyperref}
\usepackage{tikz}
\usepackage{xypic} 
\xyoption{curve}

\usetikzlibrary{decorations.markings}
\tikzstyle{vertex}=[circle, draw, inner sep=0pt, minimum size=6pt]
\tikzstyle{Vertex}=[circle, draw, inner sep=0pt, minimum size=14pt]
\tikzstyle{Vertexc}=[circle, draw, inner sep=0pt, minimum size=14pt, fill=blue!30]
\tikzstyle{vertexc}=[circle, draw, inner sep=0pt, minimum size=6pt, fill=red!40]
\tikzstyle{vertexcg}=[circle, draw, inner sep=0pt, minimum size=6pt, fill=green!70!black]

\newcommand{\beq}{\begin{equation*}}
\newcommand{\eeq}{\end{equation*}}
\newcommand{\ba}{\begin{align*}}
\newcommand{\ea}{\end{align*}}

\newcommand{\matbegin}[1]{\left (  \begin{array} {#1} }
\newcommand{\matend}{ \end{array} \right ) }

\newcommand{\avn}[1]{\textbf{A\!v}_n(#1)}
\newcommand{\ivn}[1]{\textbf{I\!v}_n(#1)}

\newcommand{\fp}{\mathrm{fp}}

\begin{document}

\title[Fixed Point Biased Involutions]{Limit Theorems for Fixed Point Biased Pattern Avoiding Involutions}

\author[Jungeun~Park]{ \ Jungeun~Park}
\address{Department of Mathematical Sciences, University of Delaware}
\email{jungeun@udel.edu}

\author[Douglas~Rizzolo]{ \ Douglas~Rizzolo}
\address{Department of Mathematical Sciences, University of Delaware}
\email{drizzolo@udel.edu}

\vskip1.3cm

\begin{abstract}
We study fixed point biased involutions that avoid a pattern.  For every pattern of length three we obtain limit theorems for the asymptotic distribution of the (appropriately centered and scaled) number of fixed points of a random fixed point biased involution avoiding that pattern.  When the pattern being avoided is either $321$, $132$, or $213$, we find a phase transition depending on the strength of the bias.  We also obtain a limit theorem for the distribution of fixed points when the pattern is $123\cdots k(k+1)$ for any $k$ and partial results when the pattern is $(k+1)k\cdots 321$. 
\end{abstract}

\maketitle
\section{Introduction}
If $\pi$ and $\sigma$ are permutations of $[n]=\{1,\dots, n\}$ and $[m]$ respectively with $n>m$, then $\pi$ is said to contain $\sigma$ if there exist indices $i_1<\dots< i_m$ such that for all $j<k$, $\pi(i_j)<\pi(i_k)$ if and only if $\sigma(j)<\sigma (k)$.  The permutation $\pi$ is said to avoid $\sigma$ if it does not contain $\sigma$.  Given a permutation $\sigma$, we let $\avn{\sigma}$ be the set of permutations of $[n]$ that avoid $\sigma$ and $\ivn{\sigma}$ be the subset of $\avn{\sigma}$ of involutions.  If $\pi$ is a permutation of $[n]$, we call $|\pi|=n$ the length of the permutation.

Recently there has been considerable interest in the study of random pattern avoiding permutations \cite{albert2024logical, bassino2018brownian, FRL, hrs1, janson2014patterns, consecutivepatterns, borga2022scaling}, and the fixed points of pattern avoiding permutations \cite{CP25, drs, mrs, hrs2, hrs3, elizalde2004bijections}. Pattern avoiding involutions specifically have drawn attention \cite{elizalde2004bijections, mrs, simionschmidt, bona2016pattern, drs} and are interesting because they shed light on the relationship between pattern avoidance and the structure of the symmetric group.  In \cite{mrs}, the asymptotic distribution of the number of fixed points for uniformly random pattern avoiding involutions was studied for patterns of length $3$ and monotone patterns.  In this paper we extend these results to fixed point biased involutions. 

Fixed point biased measures on pattern avoiding permutations were recently introduced in \cite{CP25}.  For a permutation $\pi$ let $\fp(\pi)$ be the number of fixed points in $\pi$ and $q\in (0,\infty)$, \cite{CP25} introduced the measures $\tilde \P^{\sigma,q}_{n}$ on $\avn{\sigma}$ defined by
\[ \tilde \P^{q,\sigma}_{n}(\pi) \propto q^{\fp(\pi)} \]
and studied the distribution of $\fp(\pi)$ under $\tilde \P^{q,\sigma}_{n}$ as $n\to\infty$ when $|\sigma|=3$.  They were able to obtain limiting distributions except in the case $\sigma \in \{231, 312\}$. 

We study the analogous problem for fixed point biased involutions.  We define measures $\P^{\sigma,q}_{n}$ on $\ivn{\sigma}$ by
\[ \P^{q,\sigma}_{n}(\pi) \propto q^{\fp(\pi)} \]
and study the distribution of $\fp(\pi)$ under $\P^{q,\sigma}_{n}$ as $n\to\infty$ when either $|\sigma|=3$ or $\sigma$ is a monotone pattern.  The asymptotic distribution of $\fp(\pi)$ under $\P^{1,\sigma}_{n}$, which is the uniformly random case, for these choices of $\sigma$ was studied in \cite{mrs}.  Our results extend those results to the fixed point biased case.  In contrast to \cite{CP25}, we are able to obtain results for all patterns of length $3$.  Similar to \cite{CP25}, when $\sigma\in \{321, 132, 213\}$ we find that there is a phase transition in the limiting distribution from Negative Binomial, to Rayleigh, to normal depending on the strength of the bias, see Theorem \ref{thm phase} below, however, the location of the phase transition is different. 

\section{Main Results}
The following four theorems are our main results.  In several of them random variables common distributions conditioned on their parity arise.  In order to simplify their statements we introduce the following notation: If $X$ is a $\Z$-valued random variable, we let $X_{even}$ be distributed like $X$ conditioned to be even and we let $X_{odd}$ be distributed like $X$ conditioned to be odd.

\begin{theorem}\label{thm inc}
Fix $k\geq 1$, $q>0$, $\sigma =123\cdots k(k+1)$. Suppose that $\Pi_n$ is a random element of $\ivn{\sigma}$ distributed according to $ \P^{q,\sigma}_{n}$ and suppose that $X^q$ has a binomial distribution with parameters $k$ and $q/(1+q)$. Then
\[ \fp(\Pi_{2n}) \to_d X^q_{even} \quad \textrm{and}\quad  \fp(\Pi_{2n-1}) \to_d X^q_{odd}.\]
\end{theorem}

\begin{theorem} \label{thm dec}
Fix $k\geq 1$ even, $0<q\leq 1$, $\sigma =(k+1)k\cdots 321$.  Let $q(n) = q^{\sqrt{k/n}}$. Suppose that $\Pi_n$ is a random element of $\ivn{\sigma}$ distributed according to $ \P^{q(n),\sigma}_{n}$.  Let $M$ be a random matrix drawn from the $k\times k$ Gaussian Orthogonal Ensemble conditioned to have trace $0$ and let $\Lambda_1\geq \Lambda_2  \geq \cdots \geq \Lambda_k$ be its eigenvalues.  Define
\[ S_k = \sum_{j=1}^k (-1)^{j+1} \Lambda_j,\]
and let $\P_k$ be the distribution of $S_k$.  Let $X_k$ be a random variable with distribution $\bP_k$ that is absolutely continuous with respect to $\P_k$ with density 
\[ \bP_k(ds) =\frac{q^{s}}{ \E\left[ q^{S_k}\right]} d\P_k(ds).\] 
Then 
\[ \sqrt{\frac{k}{n}} \fp(\Pi_n) \to_d X_k. \]
\end{theorem}

It would be interesting to have a similar result when $k$ is odd, but in that case $\fp(\Pi_n)$ needs to be recentered, in addition to being rescaled, which creates challenges in our proof because our proof relies on the non-negativity of $\fp(\Pi_n)$. Note that Theorem \ref{thm dec} is qualitatively different from Theorem \ref{thm inc} and the results of \cite{CP25} because $q(n)\to 1$ as $n\to \infty$ as opposed to being fixed.  It would be interesting to investigate what happens when $q$ is fixed, at the moment we are only able to do this in the case when $k=2$.

\begin{theorem}\label{thm phase}
Fix $\sigma\in \{321, 132, 213\}$. Suppose that $\Pi_n$ is a random element of $\ivn{\sigma}$ distributed according to $\P^{q,\sigma}_n$.
\begin{enumerate}
\item If $0<q<1$ then 
\[ \fp(\Pi_{2n})\to_d N_{even}  \quad \textrm{and}\quad  \fp(\Pi_{2n+1})\to_d N_{odd}  \]
where $N$ is distributed like a Negative Binomial random variable with parameters $p=1-q$ and $r=2$, that is $\P(N=k) =k q^k(1-q)^2$ for $k\in \{0,1,2,\dots\}$. 
\item If $q=1$ then 
\[ \sqrt{\frac{1}{n}} \fp(\Pi_n) \to_d \frac{1}{\sqrt{2}}X_2,\]
where $X_2$ is as defined in Theorem \ref{thm dec}.  In fact, $X_2/\sqrt{2}$ has a Rayleigh$(1)$ distribution.
\item If $q>1$ then
\[ \frac{ \fp(\Pi_n) - \frac{q^2-1}{q^2+1}n}{\frac{2q}{q^2+1}\sqrt{n}} \rightarrow_d Z\]
where $Z$ has a standard normal distribution.
\end{enumerate}
\end{theorem}

For $\sigma \in \{321, 132, 213\}$, a similar phase transition was observed in \cite{CP25} for fixed point biased permutations (not involutions), but in that case the phase transition happened at $q=3$.  In contrast, in our case the uniform case $q=1$ is the critical case.  Indeed, such phase transitions are universal in $q$-biased combinatorial structures whose generating functions satisfy a certain composition relation \cite{ban1,ban2}.  However, our results are not a direct consequence of \cite{ban1,ban2} as their assumptions do not allow for different subsequential limits as in Part (a).

\begin{theorem}\label{thm 231}
Fix $\sigma \in \{231, 312\}$.  Suppose that $\Pi_n$ is a random element of $\ivn{\sigma}$ distributed according to $\P^{q,\sigma}_n$.  Then
\[ \frac{\fp(\Pi_n) - \frac{q}{\sqrt{q^2+8}}n}{ \sqrt{\frac{8q}{(8+q^2)^{3/2}}n} } \rightarrow_d Z\]
where $Z$ has a standard normal distribution.
\end{theorem}

\section{Proofs}
An important observation that we will use in our proofs is that $ \P^{q,\sigma}_{n}$ is absolutely continuous with respect to $ \P^{1,\sigma}_{n}$ density with
\begin{equation}\label{eq dens} \P^{q,\sigma}_{n}(d\pi) = \frac{q^{\fp(\pi)}}{ \E^{1,\sigma}_n \left[q^{\fp(\Pi)}\right]} \P^{1,\sigma}_{n}(d\pi).\end{equation}

The first few results follow quickly from Equation \eqref{eq dens}, the results of \cite{mrs}, and the continuous mapping theorem.  The proofs of Theorems \ref{thm phase} and \ref{thm 231} require more calculation, but are based on standard techniques from analytic combinatorics.

\begin{proof}[Proof of Theorem \ref{thm inc}]
First consider when $q=1$.  In this case \cite[Theorem 2]{mrs} shows that
\[\P^{1,\sigma}_{2n}(\fp(\Pi) = i) \to \begin{cases} \frac{1}{2^{k-1}}{ k\choose i} & i \textrm{ is even}\\ 0& i \textrm{ is odd}\end{cases} \]
and
\[\P^{1,\sigma}_{2n+1}n(\fp(\Pi) = i) \to \begin{cases} \frac{1}{2^{k-1}}{ k\choose i} & i \textrm{ is odd}\\ 0& i \textrm{ is even}.\end{cases} \]
Direct calculation shows that if $X^1$ has a Binomial distribution with parameters $k$ and $q/(q+1)=1/2$, then
\[ \P(X^1=i | X^1 \textrm{ is even}) = \begin{cases} \frac{1}{2^{k-1}}{ k\choose i} & i \textrm{ is even}\\ 0& i \textrm{ is odd}\end{cases} \]
while
\[ \P(X^1=i | X^1 \textrm{ is odd})=\begin{cases} \frac{1}{2^{k-1}}{ k\choose i} & i \textrm{ is odd}\\ 0& i \textrm{ is even}.\end{cases} \]
Thus the result holds for $q=1$.  

We use Equation \eqref{eq dens} to extend this to the general case. Note that $\fp(\Pi_{n}) \leq k$ (in particular, it is bounded) since the fixed points of $\Pi_n$ form an increasing subsequence of $\Pi_n$ and length of the longest increasing subsequence of $\Pi_n$ is at most $k$ since $\Pi_n$ avoids $123\cdots k(k+1)$.  Since $f(x)=q^x$ is bounded on $\{0,1,\dots, k\}$ the Continuous Mapping Theorem  \cite[Theorem 5.27]{Kal} shows that
\[ \E^{1,\sigma}_{2n} \left[q^{\fp(\Pi)}\right] \to \E\left[q^{X^1_{even}}\right] \quad \textrm{and}\quad \E^{1,\sigma}_{2n+1} \left[q^{\fp(\Pi)}\right] \to \E\left[q^{X^1_{odd}}\right].\]
Consequently if follows from Equation \eqref{eq dens} that 
\[ \P^{q,\sigma}_{2n}(\fp(\Pi) = i) = \frac{q^i}{\E^{1,\sigma}_{2n} \left[q^{\fp(\Pi)}\right]} \P^{1,\sigma}_{2n}(\fp(\Pi) = i) \to  \frac{q^i}{\E\left[q^{X^1_{even}}\right]} \P(X^1_{even}=i),\]
and
\[ \P^{q,\sigma}_{2n+1}(\fp(\Pi) = i) = \frac{q^i}{\E^{1,\sigma}_{2n+1} \left[q^{\fp(\Pi)}\right]} \P^{1,\sigma}_{2n+1}(\fp(\Pi) = i) \to  \frac{q^i}{\E\left[q^{X^1_{odd}}\right]} \P(X^1_{odd}=i).\]
A direct calculation now shows that
\[ \frac{q^i}{\E\left[q^{X^1_{even}}\right]} \P(X^1_{even}=i) = \P(X^q_{even}=i) \quad \textrm{and}\quad \frac{q^i}{\E\left[q^{X^1_{odd}}\right]} \P(X^1_{odd}=i) = \P(X^q_{odd}=i),\]
which completes the proof.
\end{proof}

\begin{proof}[Proof of Theorem \ref{thm dec}]
The claim for $q=1$ is the content of \cite[Theorem 1]{mrs}.  For $q<1$ we use Equation \eqref{eq dens}.  Since $q <1$, $f(x)=q^x$ for $x\geq 0$ it is bounded and continuous.  The Continuous Mapping Theorem \cite[Theorem 5.27]{Kal} therefore implies that if $g:[0,\infty) \to \R$ is any bounded and continuous function then 
\[ \E^{1,\sigma}_n \left[g\left(\sqrt{\frac{k}{n}}\fp(\Pi)\right)q^{\sqrt{\frac{k}{n}}\fp(\Pi)}\right] \to \E( g(S_k)q^{S_k}).\]
Applying this in Equation  \eqref{eq dens} with $g\equiv 1$ in the denominator shows that
\[ \E^{q(n),\sigma}_n g\left(\sqrt{\frac{k}{n}}\fp(\Pi)\right) = \frac{ \E^{1,\sigma}_n \left[g\left(\sqrt{\frac{k}{n}}\fp(\Pi)\right)q^{\sqrt{\frac{k}{n}}\fp(\Pi)}\right] }{ \E^{1,\sigma}_n \left[q^{\sqrt{\frac{k}{n}}\fp(\Pi)}\right] } \to \frac{\E( g(S_k)q^{S_k})}{\E(q^{S_k})},\]
which completes the proof.
\end{proof}

\begin{proof}[Proof of Theorem \ref{thm phase}]
When $q=1$ this is a direct consequence of Theorem \ref{thm dec}.  

For $q\neq 1$, we will approach the problem using analytic combinatorics.  From \cite{mrs}, the bivariate generating function of involutions avoiding $\sigma$ marked by their number of fixed points is
\[ G(z,q) = \sum_n \sum_{\pi \in \ivn{\sigma}} q^{\fp(\pi)} z^n = \frac{2}{1-2qz+\sqrt{1-4z^2}}.\]
Recall that $[z^n]G(z,q) $ denotes the coefficient of $z^n$ when treating $G(z,q)$ as a formal power series in $z$, that is, 
\[ G(z,q) = \sum_{n=0}^\infty ([z^n]G(z,q)) z^n.\]
To understand the asymptotics of this function, we must understand its set of singularities.  Due to the square root function, there are singularities at $z=\pm 1/2$.  Furthermore, we see that subject to $q>0$ and $|z|\leq 1/2$, the equation
\[ 1-2qz+\sqrt{1-4z^2} =0,\]
only has a solution for $z$ when $q\geq 1$, in which case the solution is 
\[ \zeta =  \frac{q}{q^2+1}.\]
Note that when $q=1$, $|\zeta|=1/2$ while $|\zeta|<1/2$ when $q>1$.  Thus the dominant singularity of $G$ is at $z=\pm 1/2$ for $q\leq 1$ and at $z=\zeta$ for $q>1$.  The singularity structure is the most complicated when $q=1$, but fortunately for us, we have already addressed this case and need only work in the cases when $q\neq 1$.

Direct calculation shows that the probability generating function of $\fp(\Pi_n)$ is 
\[ p_n(u) =  \E^{q,\sigma}_n [ u^{\fp(\Pi_n)}]  = \frac{ [z^n]G(z,uq) }{ [z^n]G(z,q)}. \]

We now consider the case when $q<1$.  For any $0<\eta<1$, we define $H_\eta(z) = G(z,\eta)$.  Direct calculation shows that at $z=1/2$
\[ H_\eta(z)  - \frac{2}{1-\eta}  \sim \frac{-2\sqrt{2}}{(1-\eta)^2} \sqrt{1-2z}\]
and at $z=-1/2$
\[ H_\eta(z)  - \frac{2}{1+\eta}  \sim \frac{-2\sqrt{2}}{(1+\eta)^2} \sqrt{1+2z}.\]
Thus it follows from \cite[Theorem VI.5]{FlSe09} (and the remark thereafter) that 
\[ [z^n]H_{\eta}(z) \sim  \frac{2\sqrt{2}}{(1-\eta)^2}  2^n \left(\frac{1}{2\sqrt{\pi n^3}}\right) + \frac{2\sqrt{2}}{(1+\eta)^2} 2^n (-1)^n \left(\frac{1}{2\sqrt{\pi n^3}}\right). \]
Consequently,
\[  p_n(u) = \frac{[z^n]H_{uq}(z)}{[z^n]H_{q}(z)} \sim \frac{\frac{1}{(1-uq)^2}  + (-1)^n\frac{1}{(1+uq)^2}   }{\frac{1}{(1-q)^2}  +(-1)^n \frac{1}{(1+q)^2} }\]
Consequently
\[ p_{2n}(u) \to \frac{\frac{1}{(1-uq)^2}  + \frac{1}{(1+uq)^2}   }{\frac{1}{(1-q)^2}  + \frac{1}{(1+q)^2} } . \]
The function on the right hand side is the probability generating function of a Negative Binomial distribution with parameters $p=1-q$ and $r=2$ conditioned to be even.  Similarly
\[ p_{2n+1}(u) \to \frac{\frac{1}{(1-uq)^2}  - \frac{1}{(1+uq)^2}   }{\frac{1}{(1-q)^2}  - \frac{1}{(1+q)^2} } ,\]
and the function on the right hand side is the probability generating function of a Negative Binomial distribution with parameters $p=1-q$ and $r=2$ conditioned to be odd.

For $q>1$, we check the hypotheses of \cite[Theorem IX.9]{FlSe09}.  In order for our notation to match that of \cite[Theorem IX.9]{FlSe09}, we define $F(z,u) = G(z,uq)$, so that
\[ p_n(u) =\frac{ [z^n]F(z,u) }{ [z^n]F(z,1)}.\]
Note that there is an $r$ such that $F(z,1)$ is meromorphic on a ball of radius $r$ centered at the origin whose only pole is a simple pole at $\rho = q/(q^2+1)$.  By rationalizing the denominator, we can write
\[ F(z,u) = \frac{1-2uqz-\sqrt{1-4z^2}}{2z(z(u^2q^2+1)-uq)} = \frac{\frac{1-2uqz-\sqrt{1-4z^2}}{2z}}{z(u^2q^2+1)-uq}.\]
Letting 
\[ B(z,u) = \frac{1-2uqz-\sqrt{1-4z^2}}{2z} \quad \textrm{and} \quad C(z,u)=z(u^2q^2+1)-uq, \]
we have that $F(z,u) = B(z,u)/C(z,u)$.  Note that the singularity of $B$ at the origin is removable, so in fact $B$ defines a bivariate analytic function on $|z|<1/2$.

Letting
\[ \rho(u) = \frac{uq}{u^2q^2+1}\]
we note that $C(\rho(u),u)=0$.  Define $f(u) = \rho(1)/\rho(u)$.  By direct calculation, we find that
\[ f'(1) = \frac{q^2-1}{q^2+1} \quad \textrm{and} \quad f''(1)+f'(1)-(f'(1))^2 = \frac{4q^2}{(q^2+1)^2}.\]
From \cite[Theorem IX.9]{FlSe09} it follows that 
\[  \E^{q,\sigma}_n [ \fp(\Pi_n)] = \frac{q^2-1}{q^2+1}n+O(1) \quad \textrm{and} \quad \Var(\fp(\Pi_n)) = \frac{4q^2}{(q^2+1)^2}n+O(1),\]
and 
\[ \frac{ \fp(\Pi_n) - \frac{q^2-1}{q^2+1}n}{ \frac{2q}{q^2+1}\sqrt{n}} \to_d Z\]
where $Z$ has a standard normal distribution.
\end{proof}

\begin{proof}[Proof of Theorem \ref{thm 231}]
From \cite{mrs}, the bivariate generating function of involutions avoiding $\sigma$ marked by their number of fixed points is
\[ G(z,q) = \sum_n \sum_{\pi \in \ivn{\sigma}} q^{\fp(\pi)} z^n = \frac{1-z^2}{1-2z^2-qz}.\]
Direct calculation shows that the probability generating function of $\fp(\Pi_n)$ is 
\[ p_n(u) =  \E^{q,\sigma}_n [ u^{\fp(\Pi_n)}]  = \frac{ [z^n]G(z,uq) }{ [z^n]G(z,q)}. \]
To prove our result, we check the hypotheses of \cite[Theorem IX.9]{FlSe09}.  In order for our notation to match that of \cite[Theorem IX.9]{FlSe09}, we define $F(z,u) = G(z,uq)$, so that
\[ p_n(u) =\frac{ [z^n]F(z,u) }{ [z^n]F(z,1)}.\]
Letting $B(z,u) = (z^2-1)/2$ and $C(z,u) = z^2+(quz-1)/2$, we have that $F(z,u) = B(z,u)/C(z,u)$ where $B$ and $C$ are analytic.  Furthermore, letting
\[ \rho(u) =  \frac{-qu+ \sqrt{ q^2u^2 + 8}}{4} \quad \textrm{and}\quad r(u) =  \frac{-qu- \sqrt{ q^2u^2 + 8}}{4}\]
 we see that for $u$ close to $1$, $0<\rho(u)<|r(u)|$ and $C(z,u)$ factors as
\[ C(z,u) = \left( z -\rho(u)\right) \left( z- r(u)\right).\]
Consequently $C(z,1)$ has a simple root at $\rho(1)$ and $F(z,1)$ is meromorphic on a ball of radius $\sqrt{q^2+8}/4$ centered at the origin whose only singularity in that ball is a simple pole at $\rho(1)$.  Furthermore, note that $\rho(u)$ is a non-constant function that is analytic near $u=1$ and $C(\rho(u),u)=0$.  We may also directly verify that
\[  \partial_z C(\rho(1),1)\cdot\partial_uC(\rho(1),1)  \neq 0\]
and that, letting $f(u) = \rho(1)/\rho(u)$, we have that 
\[ f''(1)+f'(1) - (f'(1))^2 = \frac{8q}{(8+q^2)^{3/2}} \neq 0.\]
Consequently, \cite[Theorem IX.9]{FlSe09} implies that 
\[ \E[\fp(\Pi_n)] = f'(1) n+ O(1) = \frac{q}{\sqrt{q^2+8}} n+ O(1) \quad \textrm{and} \quad \Var(\fp(\Pi_n)) = \frac{8q}{(8+q^2)^{3/2}} n+O(1),\]
and
\[ \frac{\fp(\Pi_n) - \frac{q}{\sqrt{q^2+8}}n}{ \sqrt{\frac{8q}{(8+q^2)^{3/2}}n} } \rightarrow_d Z\]
where $Z$ has a standard normal distribution.
\end{proof}

\section{Open Problems}
Our results leave several open problems that would be interesting to investigate.

\begin{question}
Can the analogue of Theorem 2 be proved when $k$ is odd, or for $q>1$?
\end{question}

In both cases, the problem is that $f(x)=q^x$ is not bounded on $\R$ for any $q\neq 1$ or on $[0,\infty)$ for $q>1$.  When $k$ is odd, one needs to consider all of $\R$ because \cite[Theorem 1]{mrs} implies, when $q=1$, that
\[ \sqrt{\frac{k}{n}} \left( \Pi_n - \frac{n}{k}\right) \rightarrow_d S_k.\]
Due to the recentering the sequence and limit can be positive or negative.  Thus, in these cases, one would need to prove convergence of the relevant exponential moments directly rather than relying on the continuous mapping theorem as we have.

\begin{question}
What happens for $(k+1)k\cdots 321$-avoiding permutations when $q$ is fixed?
\end{question}

In Theorem 2, we took $q(n) = q^{\sqrt{k/n}}$ so that we could apply the results of \cite[Theorem 1]{mrs}, but what happens if $q(n)=q$ is fixed and does not depend on $n$? It seems that \cite{mrs} does not help in this context.

\begin{question}
In Theorem 3, what happens if $q=q(n) \to 1$?
\end{question}

This is the analogue of Theorem 2, but for $\sigma\in \{321, 132, 213\}$.  One would expect the limit to depend on the rate at which $q(n)\to 1$, and whether $q(n)$ is increasing or decreasing to $1$. 
 
\begin{question}
What happens in the non-pattern avoiding case?  The is, what happens for a uniformly random fixed point biased involution?
\end{question}

In this case generating functions apply, but the relevant generating functions are entire.  As a result, singularity analysis does not apply and one would need to apply saddle point methods.  In the unbiased case ($q=1$), appropriately renormalized the number of fixed points converges in distribution to a Gaussian random variable.  We conjecture that the same holds in general.

\bibliographystyle{plain}
\bibliography{pattern.bib}

\end{document}